\documentclass{amsart}
\usepackage{amssymb}
\usepackage{amsfonts}
\usepackage{amscd}

\newtheorem{theorem}{Theorem}

\newtheorem*{mainlemma}{Main Lemma}
\newtheorem{corollary}{Corollary}

\newtheorem{definition}{Definition}
\newtheorem{example}{Example}

\newtheorem{lemma}{Lemma}

\newtheorem{problem}{Problem}
\newtheorem{proposition}{Proposition}
\newtheorem{remark}{Remark}

\numberwithin{equation}{section}

\input{tcilatex}

\begin{document}
\title{Flats and Submersions in Non-Negative Curvature}
\author{Curtis Pro}
\address{Department of Mathematics\\
University of California, Riverside}
\email{pro@math.ucr.edu}
\author{Frederick Wilhelm}
\address{Department of Mathematics\\
University of California, Riverside}
\email{fred@math.ucr.edu}

\begin{abstract}
We find constraints on the extent to which O'Neill's horizontal curvature
equation can be used to create positive curvature on the base space of a
Riemannian submersion. In particular, we study when K. Tapp's theorem on
Riemannian submersions of compact Lie groups with bi-invariant metrics
generalizes to arbitrary manifolds of non-negative curvature.
\end{abstract}

\maketitle

Until very recently all examples of compact, positively curved manifolds
were constructed as the image of a Riemannian submersion of a Lie group with
a bi-invariant metric (\cite{Dear,GrovVerdZil,PetWilh2}). Earlier
constructions of positive curvature in \cite{AlfWal,Baz,Berg1}, and \cite%
{Esch1,Esch2,Esch3} combined the fact that Lie groups with bi-invariant
metrics are non-negatively curved with the so called Horizontal Curvature
Equation, 
\begin{equation*}
\mathrm{sec}_{B}\left( x,y\right) =\mathrm{sec}_{M}\left( \tilde{x},\tilde{y}%
\right) +3\left\vert A_{\tilde{x}}\tilde{y}\right\vert ^{2}
\end{equation*}%
\cite{Gray,O'Neill}. Here $\pi :M\rightarrow B$ is a Riemannian submersion, $%
\left\{ x,y\right\} $ is an orthonormal basis for a plane in a tangent space
to $B,$ $\left\{ \tilde{x},\tilde{y}\right\} $ is a horizontal lift of $%
\left\{ x,y\right\} ,$ and $A$ is the \textquotedblleft integrability
tensor\textquotedblright\ for the horizontal distribution---that is, 
\begin{equation*}
A_{\tilde{x}}\tilde{y}\equiv \frac{1}{2}[\tilde{X},\tilde{Y}]^{\mathrm{vert}}
\end{equation*}%
where $\tilde{X}$ and $\tilde{Y}$ are arbitrary extensions of $\tilde{x}$
and $\tilde{y}$ to horizontal vector fields.

Since the Horizontal Curvature Equation decomposes $\mathrm{sec}_{B}\left(
x,y\right) $ into the sum of two non-negative quantities, we see immediately
that Riemannian submersions preserve non-negative curvature. In addition, if 
\emph{either} term on the right is positive, then $\mathrm{sec}_{B}\left(
x,y\right) >0.$ Naively, one might expect positively curved examples to be
constructed by exploiting the full power of the Horizontal Curvature
Equation; however, a survey of the examples shows that this has never been
done. In the context in which the examples in \cite{AlfWal,Baz,Berg1,
Esch1,Esch2,Esch3}, and \cite{Wal} were constructed, it is impossible for a
Riemannian submersion to create positive curvature via the $A$--tensor
alone. In fact, in \cite{Tapp} Tapp shows

\begin{theorem}[Tapp]
\label{tappt} Let $\pi :G\rightarrow B$ be a Riemannian submersion of a
compact Lie group with a bi-invariant metric. Then

\begin{description}
\item[1] Every zero-curvature plane of $B$ exponentiates to a flat (meaning
a totally geodesic immersion of $\mathbb{R}^{2}$ with a flat metric)$,$ and

\item[2] Every horizontal zero-curvature plane of $G$ projects to a
zero-curvature plane of $B.$
\end{description}
\end{theorem}

In the case of bi-quotients of Lie groups, this is a consequence of an
equation in \cite{GromMey}. This was first observed explicitly in \cite%
{Wilk1}.

Examples \ref{FB1} and \ref{FB2} (below) show that the theorem fails if the
Lie group $G$ is replaced by an arbitrary, compact, non-negatively curved
Riemannian manifold $M$. The inhomogeneous metrics of these examples have
zero-planes whose exponentials are locally, but not globally, flat.

Recall, if $\sigma $ is a zero-curvature plane in a Lie group $G$ with
bi-invariant metric, then $\exp (\sigma )$ is a (globally) flat submanifold
of $G$. So it is natural to ask about the extent to which Tapp's theorem
holds if $\sigma $ is assumed to be a horizontal zero-curvature plane whose
exponential image is a flat submanifold of $M$. More formally, we pose:

\begin{problem}
If $\pi :M\rightarrow B$ is a Riemannian submersion of a compact,
non-negatively curved manifold $M$ and $\sigma $ is a horizontal
zero-curvature plane in $M$ such that $\exp (\sigma )$ is a flat
submanifold, does it follow that $\pi _{\ast }(\sigma )$ is a zero-curvature
plane in $B$?
\end{problem}

We emphasize that the given flat is not assumed to be globally horizontal.

The following easy consequence of Lemma 1.5 in \cite{StraWals} shows that an
affirmative answer to our problem implies that both $M$ and $B$ have a lot
of additional structure.

\begin{theorem}
\label{t1} Let $\pi :M\rightarrow B$ be a Riemannian submersion of complete,
non-negatively curved manifolds. Let $\sigma $ be a zero-curvature plane in $%
B$ and $\tilde{\sigma}$ a horizontal lift of $\sigma $ so that $\exp (\tilde{%
\sigma})$ is a flat in $M$. Then

\begin{description}
\item[1] The plane $\sigma $ exponentiates to a flat in $B$, and

\item[2] Every horizontal lift of $\sigma $ exponentiates to a horizontal
flat in $M$.
\end{description}
\end{theorem}

In Theorem \ref{t1}, we do not require that $M$ is compact; on the other
hand, without compactness, the answer to Problem 1 is \textquotedblleft
no\textquotedblright , even when $M$ is a Lie group.

\begin{example}
\label{Cheeg}Let $(\mathbb{R}^{2},\bar g)$ be the Cheeger deformation of $%
\mathbb{R}^{2}$ obtained from the standard $S^{1}$ action on $\mathbb{R}^{2}$%
. Let $s$ and $g$ be the usual metrics on $S^{1}$ and $\mathbb{R}^{2}$,
respectively. Recall that $\bar g$ is defined so that the quotient map, 
\begin{equation*}
Q:(S^{1}\times \mathbb{R}^{2},s+g)\rightarrow (\mathbb{R}^{2},\bar g)
\end{equation*}%
given by $Q(z,q)=\bar{z}q$ is a Riemannian submersion. This new metric is
positively curved and is a paraboloid asymptotic to a cylinder of radius 1.
All horizontal planes have zero curvature, but each projects to a positively
curved plane. So positive curvature is created via the $A$-tensor alone.
\end{example}

\begin{example}
\label{FB1} (Fish Bowl) Let $\psi :\left[ 0,\pi \right] \longrightarrow 
\mathbb{R}$ be a smooth, concave-down function that satisfies 
\begin{equation*}
\psi \left( t\right) =\left\{ 
\begin{array}{cc}
t & \text{for }t\in \left[ 0,\frac{\pi }{4}\right] \\ 
\pi -t & \text{for }t\in \left[ \frac{3\pi }{4},\pi \right]%
\end{array}%
\right.
\end{equation*}%
Consider the warped product metric 
\begin{equation*}
g_{\psi }=dt^{2}+\psi ^{2}d\theta ^{2}
\end{equation*}%
on $S^{2}=\left[ 0,\pi \right] \times _{\psi }S^{1}$. As before, $S^{1}$
acts isometrically on $\left( S^{2},g_{\psi }\right) ,$ so we get a
Riemannian submersion 
\begin{equation*}
\left( S^{2},g_{\psi }\right) \times S^{1}\longrightarrow \left( S^{2},\bar{g%
}_{\psi }\right) ,
\end{equation*}%
where $\bar{g}_{\psi }$ is the metric induced by the submersion. Notice that 
$\left( S^{2},g_{\psi }\right) \times S^{1}$ is flat in a neighborhood of
the set $\left\{ 0,\pi \right\} \times S^{1},$ but, as in Example \ref{Cheeg}%
, $\left( S^{2},\bar{g}_{\psi }\right) $ is positively curved in the image
of this neighborhood. If, in addition, 
\begin{equation*}
\psi ^{\prime \prime }|_{\left( \frac{\pi }{4},\frac{3\pi }{4}\right) }<0,
\end{equation*}%
then $\left( S^{2},\bar{g}_{\psi }\right) $ is positively curved. This shows
that even in the compact case, the $A$-- tensor can be responsible for
creating positive curvature and that conclusion 2 of Tapp's Theorem fails
for arbitrary Riemannian submersions of compact, nonnegatively curved
manifolds.
\end{example}

\begin{example}
\label{FB2} To see how conclusion 1 of Tapp's theorem can fail to hold,
choose $\psi $ in the previous example to be constant in a neighborhood of $%
\pi /2$. This makes $\left( S^{2},g_{\psi }\right) $ isometric to a flat
cylinder near a neighborhood of the equator. In the Cheeger deformed metric,
the image of this region is a smaller flat cylinder. Since the base, $\left(
S^{2},\bar{g}_{\psi }\right) ,$ is not flat, we have zero--curvature planes
near the equator that do not exponentiate to flats.
\end{example}

If we assume the fibers of the submersion are totally geodesic, then, even
in the non-compact case, the conclusion of Tapp's theorem holds.

\begin{theorem}
\label{tgft} Let $\pi :M\rightarrow B$ be a Riemannian submersion of
complete, non-negatively curved manifolds with totally geodesic fibers. Let $%
\tilde{\sigma}$ be a horizontal zero-curvature plane in $M$ such that $\exp (%
\tilde{\sigma})$ is a flat. Then

\begin{description}
\item[1] $\tilde{\sigma}$ projects to a zero-curvature plane $\sigma $ in $B$
that exponentiates to a flat submanifold of $B$, and

\item[2] Every horizontal lift of $\sigma $ exponentiates to a horizontal
flat in $M$.
\end{description}
\end{theorem}

We also give an affirmative answer to Problem 1 in the special case when the
submersion is induced by an isometric group action with only principal
orbits.

\begin{theorem}
\label{t2} Let a compact Lie group $G$ act by isometries on a compact,
non-negatively curved manifold $M$. Suppose all of the orbits are principal,
and let $\pi :M\rightarrow M/G$ be the induced Riemannian submersion.

Suppose $\tilde{\sigma}$ is a horizontal zero-curvature plane in $M$ such
that $\exp _{p}(\tilde{\sigma})$ is a flat. Then

\begin{description}
\item[1] $\tilde{\sigma}$ projects to a zero-curvature plane $\sigma $ in $%
M/G$ that exponentiates to a flat submanifold of $M/G$, and

\item[2] Every horizontal lift of $\sigma $ exponentiates to a horizontal
flat in $M$.
\end{description}
\end{theorem}

Example \ref{Cheeg} shows that this result does not hold if we remove the
hypothesis that $M$ is compact. On the other hand, appropriate associated
bundles also inherit this property.

\begin{corollary}
\label{Associated}Let $G$ be a compact Lie group, $P$ be compact, and $\pi
_{P}:P\rightarrow B\equiv P/G$ a principal $G$--bundle with non-negatively
curved $G$--invariant metric. Let $F$ be a non-negatively curved manifold
that carries an isometric $G$--action and $\pi :E:=P\times _{G}F\rightarrow
B $ the corresponding associated bundle with fiber $F$. Give $E$ and $B$ the
corresponding non-negatively curved metrics so that $\pi $ and $Q:P\times
F\rightarrow P\times _{G}F=E$ become Riemannian submersions.

If $\tilde{\sigma}$ is a $\pi$--horizontal zero-curvature plane in $E$ such
that $\exp _{p}(\tilde{\sigma})$ is a flat, then

\begin{description}
\item[1] $\tilde{\sigma}$ projects to a zero-curvature plane $\sigma $ in $B$
that exponentiates to a flat submanifold of $B$, and

\item[2] Every horizontal lift of $\sigma $ exponentiates to a horizontal
flat in $E$.
\end{description}
\end{corollary}

There is an abstract application of Theorem \ref{t1} in \cite{PetWilh1}. It
allows for a simplification of one of the axioms for the Orthogonal Partial
Conformal Change. There are also quite a few concrete examples of our
results in the literature that are not examples of Theorem \ref{tappt}.

\begin{example}
Grove and Ziller have shown how to lift the product metric on $S^{2}\times
S^{2}$ and Cheeger's metric on $\mathbb{C}P^{2}\#-\mathbb{C}P^{2}$ to
various principal $SO\left( k\right) $ bundles and hence to all of the
associated bundles \cite{GrovZil}. According to Lemma \ref{Stupid} (below)
the flat tori in $S^{2}\times S^{2}$ lift to flats in all of these
non-negatively curved bundles. Similarly, the flat Klein bottles in
Cheeger's $\mathbb{C}P^{2}\#-\mathbb{C}P^{2}$ must also lift to flats in all
of the non-negatively curved bundles of \cite{GrovZil}. It follows from the
construction of the metric that the principal bundles all have totally
geodesic fibers. Therefore the principal bundles give examples of Theorems %
\ref{t1}, \ref{tgft}, and \ref{t2}. The associated bundles give examples of
Theorem \ref{t1} and Corollary \ref{Associated}.
\end{example}

To prove Theorems \ref{tgft} and \ref{t2} we establish a main lemma on
holonomy fields, whose defintion we recall from \cite{GromWals}.

\begin{definition}
Given a Riemannian submersion $\pi :M\rightarrow B$ let $A$ and $T$ be the
corresponding fundamental tensors as defined in \cite{O'Neill}. A Jacobi
field $J$ along a horizontal geodesic $c:I\rightarrow M$ is said to be a
holonomy field if $J(0)$ is vertical and satisfies 
\begin{equation}
J^{\prime }(0)=A_{\dot{c}(0)}J(0)+T_{J(0)}\dot{c}(0).  \label{hfd}
\end{equation}
\end{definition}

\begin{mainlemma}
Let $\pi :M\rightarrow B$ be a Riemannian submersion of complete,
non-negatively curved manifolds so that each holonomy field is bounded. Let $%
\tilde{\sigma}$ be a horizontal zero-curvature plane in $M$ such that $\exp (%
\tilde{\sigma})$ is a flat. Then

\begin{description}
\item[1] $\tilde{\sigma}$ projects to a zero-curvature plane $\sigma $ in $B$
that exponentiates to a flat submanifold of $B$, and

\item[2] Every horizontal lift of $\sigma $ exponentiates to a horizontal
flat in $M$.
\end{description}
\end{mainlemma}

The Main Lemma is a consequence of Lemmas \ref{cl} and \ref{Stupid} (below).
These along with Theorems \ref{t1} and \ref{tgft} are proven in Section 1.
In Section 2, we prove Theorem \ref{t2} by showing that such submersions
satisfy the hypotheses of the main lemma. Corollary \ref{Associated} is also
proven in Section 2.

\noindent \textbf{Acknowledgement: }\emph{We are greatful to Owen Dearricott
for asking if we could prove Theorem \ref{tgft}.}

\section{Jacobi Fields Along Geodesics Contained In Flats}

The symmetries of the curvature tensor imply that the map $X\longmapsto
R\left( X,W\right) W$ is self-adjoint. This combined with the spectral
theorem yields the following result, which appears implictly in \cite%
{PetWilh2}.

\begin{proposition}
\label{R-0}Let $\mathrm{span}\left\{ X,W\right\} $ be a zero curvature plane
in a nonnnegatively curved manifold, then 
\begin{equation*}
R\left( X,W\right) W=R\left( W,X\right) X=0.
\end{equation*}
\end{proposition}

In a compact Lie group $G$ with bi-invariant metric, solutions to the Jacobi
equation along a geodesic $\gamma (t)$ have the form 
\begin{equation*}
J(t)=E_{0}+tF_{0}+\sum_{i=0}^{l}\left( \cos ({\sqrt{k_{i}}t})E_{i}+\sin (%
\sqrt{k_{i}}r)F_{i}\right) ,
\end{equation*}%
where $E_{i}$ and $F_{i}$ are parallel along $\gamma $ (see \cite{Munt}). We
generalize this decomposition in the following way:

\begin{lemma}
\label{jfl} Suppose $\gamma $ is a geodesic in a complete, non-negatively
curved manifold $M$, and suppose $J_{0}$ is a normal, parallel, Jacobi field
along $\gamma $, then any normal Jacobi field $J$ along $\gamma $ can be
written as 
\begin{equation}
J(t)=(a+bt)J_{0}(t)+W(t),  \label{jac}
\end{equation}%
where $a,b\in \mathbb{R}$ and $W$ and $W^{\prime }$ are perpendicular to $%
J_{0}$.
\end{lemma}

\begin{proof}
Extend $J_{0}$ to an orthonormal basis $\{J_{0},E_{2},...,E_{n-1}\}$ of
normal, parallel fields along $\gamma $. Since $J_{0}(t)$ and $\gamma
^{\prime }(t)$ span a zero-curvature plane and $M$ is non-negatively curved, 
$R(J_{0},\gamma ^{\prime })\gamma ^{\prime }=0$, by Proposition \ref{R-0}.
Therefore, if we write 
\begin{equation*}
J(t)=f(t)J_{0}(t)+\sum_{i=2}^{n-1}f_{i}(t)E_{i}(t),
\end{equation*}%
we have 
\begin{eqnarray*}
J^{\prime \prime }(t) &=&-R(J(t),\gamma ^{\prime }(t))\gamma ^{\prime }(t) \\
&=&-\sum_{i=2}^{n-1}f_{i}(t)R(E_{i},\gamma ^{\prime }(t))\gamma ^{\prime }(t)
\end{eqnarray*}%
and 
\begin{equation*}
\langle R(E_{i},\gamma ^{\prime })\gamma ^{\prime },J_{0}\rangle =\langle
R(J_{0},\gamma ^{\prime })\gamma ^{\prime },E_{i}\rangle =0
\end{equation*}%
by a symmetry of the curvature tensor. Thus $J^{\prime \prime }\perp J_{0}$.
Since $\{J_{0},E_{2},...,E_{n-1}\}$ is parallel and orthogonal, we also have 
\begin{equation*}
J^{\prime \prime }(t)=f^{\prime \prime
}(t)J_{0}(t)+\sum_{i=2}^{n-1}f_{i}^{\prime \prime }(t)E_{i}(t).
\end{equation*}%
Combining this with $J^{\prime \prime }\perp J_{0},$ we see that $f^{\prime
\prime }=0$ as claimed.

Since $W^{\prime }=\sum_{i=2}^{n-1}f_{i}^{\prime }(t)E_{i}(t),$ we also have 
$W^{\prime }\perp J_{0}.$
\end{proof}

Given a Riemannian submersion $\pi :M\rightarrow B$ , let $\mathcal{V}$ and $%
\mathcal{H}$ be the vertical and horizontal distributions. As holonomy
fields are the variational fields arising from horizontal lifts of geodesics
in $B$, they never vanish, they remain vertical, and they satisfy (\ref{hfd}%
) for all time. In fact, we can find a collection $\{J_{i}(t)\}$ of such
fields that span $\mathcal{V}$ along $c$. This description of $\mathcal{V}$
allows one to determine precisely when a field along a curve in $M$ has
values in $\mathcal{H}$. In particular, we have the following, as observed
by Tapp when $M$ is a Lie group.

\begin{lemma}
\label{bjfl} Suppose $\pi :M\rightarrow B$ is a Riemannian submersion of a
complete, non-negatively curved manifold $M$, $\gamma $ is a horizontal
geodesic in $M$, and $J_{0}$ is a parallel Jacobi field along $\gamma $ such
that $J_{0}(0)$ is horizontal. If all holonomy fields $V$ along $\gamma $
have bounded length, then $J_{0}$ is everywhere horizontal.
\end{lemma}

\begin{proof}
Let $V$ be a holonomy field. Since $V$ is always vertical, the decomposition
in Lemma~\ref{jfl} simplifies to 
\begin{equation*}
V(t)=btJ_{0}(t)+W(t).
\end{equation*}%
Since $V$ has bounded length, $b=0$ and therefore $V(t)=W(t),$ which is
perpendicular to $J_{0}$. As the collection of all holonomy fields spans the
vertical distribution along $\gamma $, the result follows.
\end{proof}

Part 1 of the main lemma is a consequence of the next result.

\begin{lemma}
\label{cl} Suppose $\pi :M\rightarrow B$ is a Riemannian submersion of a
complete, non-negatively curved manifold $M$, and all holonomy fields of $%
\pi $ have bounded length. Suppose $\tilde{\sigma}$ is a horizontal zero
curvature plane and $\exp \left( \tilde{\sigma}\right) $ is a totally
geodesic flat.

Then $\sigma :=d\pi \left( \tilde{\sigma}\right) $ has a zero curvature and $%
\exp (\sigma )$ is a totally geodesic flat submanifold of $B$.
\end{lemma}

\begin{proof}
Let $\left\{ X,Y\right\} $ be any orthonormal pair in $\tilde{\sigma}.$ Let $%
\gamma $ be the geodesic$:t\longmapsto \exp \left( tX\right) ,$ and let $J$
be the parallel Jacobi field along $\gamma $ with $J\left( 0\right) =Y.$
Then by the previous Lemma, $J\left( t\right) $ is horizontal for all $t.$
Hence $\exp \left( \tilde{\sigma}\right) $ is everywhere horizontal, and, by
assumption, a totally geodesic flat.

It follows from the Horizontal Curvature Equation that $\pi (\exp \left( 
\tilde{\sigma}\right) )$ is also flat, and from the formula for covariant
derivatives of horizontal fields it follows that $\pi (\exp \left( \tilde{%
\sigma}\right) )$ is totally geodesic. Since horizontal geodesics project to
geodesics, $\pi (\exp \left( \tilde{\sigma}\right) )=\exp (d\pi \left( 
\tilde{\sigma}\right) )=\exp (\sigma ).$ So $\exp (\sigma )$ is a totally
geodesic flat submanifold of $B$.
\end{proof}

The following lemma is probably a well known application of the Horizontal
Curvature Equation. We include it as it establishes part 2 of our main lemma
and is also used in the proof of Theorem \ref{t1}.

\begin{lemma}
\label{Stupid}Let $\pi :M\rightarrow B$ be a Riemannian submersion of a
complete, non-negatively curved manifold $M.$ Let $\sigma $ be a tangent
plane to $B$ so that $\exp (\sigma )$ is a totally geodesic flat.

Then for any horizontal lift $\tilde{\sigma}$ of $\sigma ,$ $\exp (\tilde{%
\sigma})$ is a totally geodesic flat that is everywhere horizontal.
\end{lemma}

\begin{proof}
The Horizontal Curvature Equation implies that any horizontal lift $\hat{\tau%
}$ of a plane $\tau $ tangent to $\exp (\sigma )$ satisfies 
\begin{equation*}
\mathrm{sec}_{M}\left( \hat{\tau}\right) =0\text{ and }A\left( \hat{\tau}%
\right) =0.
\end{equation*}%
In particular, the collection of all such $\hat{\tau}$s gives us an
integrable $2$-dimensional distribution that is horizontal. The vanishing $A$%
--tensor combined with our hypothesis that $\exp (\sigma )$ is totally
geodesic gives us that all the integral submanifolds of this distribution
are also totally geodesic. If $\tilde{\sigma}$ is a horizontal lift of $%
\sigma $, then it follows that $\exp (\tilde{\sigma})$ is tangent to this
distribution and hence is a totally geodesic flat that is everywhere
horizontal.
\end{proof}

We now proceed with proofs of theorems \ref{tgft} and \ref{t1}.

\begin{proof}[Proof of Theorem \protect\ref{tgft}]
When the fibers of a Riemannian submersion are totally geodesic, the $T$%
-tensor for the submersion vanishes. If $V$ is a holonomy field along a
horizontal geodesic $\gamma $, by (\ref{hfd}) we have 
\begin{equation*}
\langle V(t),V(t)\rangle ^{\prime }=2\langle V(t),V^{\prime }(t)\rangle
=2\langle V(t),T_{V(t)}\gamma ^{\prime }(t)\rangle =0,
\end{equation*}%
so $V$ has constant norm. An applicaiton of the main lemma completes the
proof.
\end{proof}

In contrast to our other results the proof of Theorem 2 does not use the
main lemma. Instead we exploit the infinitesimal geometry of the submersion.

\begin{proof}[Proof of Theorem~\protect\ref{t1}]

Let $\sigma $ be a zero-curvature plane in $B$ and $\tilde{\sigma}$ a
horizontal lift of $\sigma $ so that $\exp (\tilde{\sigma})$ is contained in
a flat of $M$. Let $\gamma $ be a geodesic in $\exp (\tilde{\sigma})$ and $%
J_{0}$ be a parallel Jacobi field along $\gamma $ such that 
\begin{equation*}
\tilde{\sigma}=\mathrm{span}\left\{ \gamma ^{\prime }\left( 0\right)
,J_{0}\left( 0\right) \right\} .
\end{equation*}

Now $A_{{\gamma }^{\prime }(0)}J_{0}(0)=0$ because $\mathrm{sec}_{M}(\tilde{%
\sigma})=\mathrm{sec}_{B}(\sigma )=0$; so for any holonomy field $V,$ we
have 
\begin{eqnarray*}
\left. \langle {J_{0}}(t),V^{\prime }(t)\rangle \right\vert _{t=0} &=&\left.
\langle J_{0}(t),A_{{\gamma }^{\prime }(t)}V(t)\rangle \right\vert _{t=0},%
\text{ since }J_{0}(0)\text{ is horizontal} \\
&=&-\left. \langle A_{{\gamma }^{\prime }(t)}J_{0}(t),V(t)\rangle
\right\vert _{t=0} \\
&=&0.
\end{eqnarray*}%
On the other hand, differentiating the right hand side of $%
V(t)=btJ_{0}(t)+W(t),$ we find 
\begin{eqnarray*}
\left. \langle {J_{0}}(t),V^{\prime }(t)\rangle \right\vert _{t=0} &=&\left.
\langle {J_{0}}(t),bJ_{0}(t)\rangle \right\vert _{t=0}+\left. \langle {J_{0}}%
(t),W^{\prime }(t)\rangle \right\vert _{t=0} \\
&=&b\left\vert J_{0}(0)\right\vert ^{2}.
\end{eqnarray*}%
Therefore $b=0$ and $V=W$, and it follows that $N:=\exp ({\tilde{\sigma}})$
is everywhere horizontal. Thus its projection, $\exp \left( \sigma \right) $%
, is a totally geodesic flat in $B.$

By Lemma \ref{Stupid}, every horizontal lift of $\sigma $ exponentiates to a
horizontal flat in $M.$
\end{proof}

\section{The Holonomy of $\protect\pi $}

In this section we prove Theorem~\ref{t2} by showing that such submersions
have bounded holomomy fields and hence satisfy the hypotheses of the main
lemma.  At the end of the section we prove Corollary \ref{Associated}.

Given a point $b\in B,$ we define the \textit{holonomy group} $\mathrm{hol}%
(b)$ to be the group of all diffeomorphisms of the fiber $\pi ^{-1}(b)$ that
occur as holonomy diffeomorphisms $h_{c}:\pi ^{-1}(b)\rightarrow \pi
^{-1}(b) $ obtained by lifting piecewise smooth loops $c$ at $b$. If $M$ is
compact, the $T$ tensor is globally bounded in norm. It follows that each
holonomy diffeomorphism $h_{c}$ is Lipschitz with Lipschitz constant
dependent only on the length of $c$ (see ~\cite{GuiWals}, Lemma 4.2). Since
this Lipschitz constant can actually depend on the length of $c,$ this is
generally not enough to conclude that the the holonomy fields are uniformly
bounded (see \cite{Tapp}, Example 6.1]).

On the other hand, if $B$ is compact and $\mathrm{hol}(b)$ is a compact,
finite-dimensional Lie group, then there is a uniform Lipschitz constant for
all of $\mathrm{hol}(b)$. Thus the holonomy fields are uniformly bounded (%
\cite{Tapp}, Proposition 6.2). So to prove theorem~\ref{t2}, it suffices to
show that $\mathrm{hol}(b)$ is a compact, finite-dimensional Lie group.

\begin{proof}[Proof of Theorem~\protect\ref{t2}]
Set $B=M/G,$ and for $p\in M,$ let $G_{p}$ denote the isotropy subgroup of $%
G $. Note that the map $f:G/G_{p}\rightarrow M$ given by $f(gG_{p})=g(p)$ is
an imbedding onto the orbit $G(p)$ of $p$. Now take any piecewise smooth
curve $c:[0,1]\rightarrow B$. The holonomy diffeomorphism 
\begin{equation*}
h_{c}:\pi ^{-1}(c(0))\rightarrow \pi ^{-1}(c(1))
\end{equation*}%
is defined by 
\begin{equation*}
h_{c}(p)=\bar{c}(1),
\end{equation*}%
where $\bar{c}$ is the unique horizontal lift of $c$ starting at $p$. By
assumption, $G$ acts isometrically on $M$, so $g\bar{c}$ is also horizontal.
Since $(g\bar{c})(1)=g(\bar{c}(1))$, we have that 
\begin{equation*}
h_{c}(gp)=gh_{c}(p).
\end{equation*}%
In other words, $h_{c}$ is $G$-equivariant.

By the above, $\mathrm{hol}(b)$ is a subgroup of the collection $\mathrm{Diff%
}_{G}(\pi ^{-1}(b))$ of all $G$-equivariant diffeomorphisms of the fiber $%
\pi ^{-1}(b)$. Take any $p\in \pi ^{-1}(b)$. Set $H\equiv G_{p},$ and
identify $\pi ^{-1}(b)$ with $G/H$. Then $\mathrm{Diff}_{G}(G/H)$ is
isomorphic to the Lie group $N(H)/H,$ where $N(H)$ is the normalizer of $H$
(see \cite{GromWals}, Lemma 2.3.3).

In \cite{Wilk2}, Wilking associates to a given metric foliation $\mathcal{F}$
the so-called \textit{dual foliation} $\mathcal{F}^{\#}$. The dual leaf
through a point $p\in M$ is defined as all points $q\in M$ such that there
is a piecewise smooth, horizontal curve from $p$ to $q$. Let $L_{p}^{\#}$ be
the dual leaf through $p$.

We shall see that for any $p\in M,$ $\mathrm{hol}(b)$ is homeomorphic to $%
L_{p}^{\#}\cap \pi ^{-1}(b).$

We have the continuous map 
\begin{equation*}
\mathrm{ev}_{p}:\mathrm{hol}(b)\rightarrow L_{p}^{\#}\cap \pi ^{-1}(b)
\end{equation*}
defined by 
\begin{equation*}
\mathrm{ev}_{p}:h_{c}\mapsto h_{c}(p).
\end{equation*}

To construct the inverse, let $q$ be in $L_{p}^{\#}\cap \pi ^{-1}(b)$. There
is a piecewise smooth, horizontal curve $\bar{c}$ from $p$ to $q$. Now $\pi
\circ \bar{c}$ is a piecewise smooth loop at $b$ and 
\begin{equation*}
h_{\pi \circ \bar{c}}(p)=q.
\end{equation*}

We therefore propose to define $\mathrm{ev}_{p}^{-1}$ by 
\begin{equation*}
\mathrm{ev}_{p}^{-1}:q\longmapsto h_{\pi \circ \bar{c}}.
\end{equation*}

To see that $\mathrm{ev}_{p}^{-1}$ is well-defined, suppose $\tilde{c}$ is
another piecewise smooth, horizontal curve from $p$ to $q$. By construction,
we have $h_{\pi \circ \bar{c}}(p)=h_{\pi \circ \tilde{c}}(p).$ Since all
holonomy diffeomorphisms are $G$-equivariant and $G$ acts transitively on $%
\pi ^{-1}(b),$ it follows that 
\begin{equation*}
h_{\pi \circ \bar{c}}=h_{\pi \circ \tilde{c}}.
\end{equation*}%
Now take a sequence of points $q_{i}\in L^{\#}\cap \pi ^{-1}(b)$ converging
to $q_{0}\in L^{\#}\cap \pi ^{-1}(b)$. There are horizontal curves $\bar{c}%
_{i}$ from $p$ to $q_{i}$ such that $h_{\pi \circ \bar{c}_{i}}(p)=q_{i}$.
Again by $G$-equivariance and the transitive action of $G$, these holonomy
diffeomorphisms are completely determined by their behavior at a point. Thus 
$h_{\pi \circ \bar{c}_{i}}\rightarrow h_{\pi \circ \bar{c}_{0}}$, and so $%
\mathrm{ev}_{p}^{-1}$ is continuous. Therefore $\mathrm{hol}(b)$ is
homeomorphic to $L^{\#}\cap \pi ^{-1}(b)$.

Since $\mathcal{F}$ is given by the orbit decomposition of an isometric
group action, the dual foliation has complete leaves (\cite{Wilk2}, Theorem
3(a)). In particular, this says $L^{\#}\cap \pi ^{-1}(b)\cong \mathrm{hol}%
(b) $ is a closed subset of the compact space $\pi ^{-1}(b)$ and hence is
also compact. It follows that $\mathrm{hol}(b)$ is closed in the Lie group $%
\mathrm{Diff}_{G}(G/H)\cong N(H)/H,$ so is a Lie subgroup of $\mathrm{Diff}%
_{G}(G/H)$. Thus $\mathrm{hol}(b)$ is a compact, finite-dimensional Lie
group.
\end{proof}

\begin{remark}
In general, $\mathrm{hol}(b)$ need not even be a Lie group, let alone a
compact Lie group \cite{Tapp}. However, it is shown in \cite{GuiWals0} that
when the fibers come from principal $G$-actions, $\mathrm{hol}(b)$ is always
a Lie group.
\end{remark}

Recall (see ~\cite{GromWals}, p.92) that if $P$ is the total space of the
principal $G$--bundle $\pi _{P}:P\rightarrow B:=P/G$ and $F$ is a manifold
that carries a $G$--action, then $G$ acts freely on the product $P\times F$.
In particular, if $P$ and $F$ have $G$-invariant metrics of non-negative
curvature, $G$ acts by isometries on the product $P\times F$. As a result,
the total space $E=P\times _{G}F:=(P\times F)/G$ of the associated bundle
inherits a metric of non-negative curvature such that the quotient map $%
Q:P\times F\rightarrow P\times _{G}F$ is a Riemannian submersion \cite{Cheeg}%
. Similarly, $B$ inherits a metric of non-negative curvature such that $\pi
_{P}:P\rightarrow B$ is a Riemannian submersion. If $\pi _{1}:P\times
F\rightarrow P$ is projection onto the first factor, the diagram

\begin{equation*}
\begin{CD} P\times F @>Q>> E\\ @V{\pi_1}VV @VV{\pi}V\\ P @>>\pi_P> B \end{CD}
\end{equation*}%
commutes and so $\pi :E\rightarrow B$ is also a Riemannian submersion.

\noindent

\begin{proof}[Proof of Corollary \protect\ref{Associated}: ]
Consider the composition 
\begin{equation*}
\pi _{P}\circ \pi _{1}:P\times F\longrightarrow B.
\end{equation*}

The holonomy fields for $\pi _{P}\circ \pi _{1}$ are the products of
holonomy fields for $\pi _{P}:P\rightarrow B$ and $\pi _{1}.$ The former are
bounded by the proof of Theorem \ref{t2}, the latter are bounded because the
fibers of $\pi _{1}$ are totally geodesic.

Now suppose that $\tilde{\sigma}$ is a horizontal zero-curvature plane for $%
\pi :E\longrightarrow B$ such that $\exp _{p}(\tilde{\sigma})$ is a flat.
Apply Lemma \ref{Stupid} to $Q:P\times F\rightarrow E$ to conclude that any
horizontal lift $\tilde{\sigma}_{P\times F}$ of $\tilde{\sigma}$
exponentiates to a ($Q$--horizontal) flat. Since the holonomy fields of $\pi
_{P}\circ \pi _{1}=\pi \circ Q$ are bounded, we can apply Lemma \ref{cl} to
conclude that $\sigma :=d\left( \pi \circ Q\right) \left( \tilde{\sigma}%
_{P\times F}\right) =d\pi \left( \tilde{\sigma}\right) $ is a zero plane
that exponentiates to a flat. Applying Lemma \ref{Stupid} to $\pi
:E\rightarrow B$ we conclude that every horizontal lift of $\sigma $ is a
horizontal flat.
\end{proof}

\begin{remark}
Combining the Main Lemma with the concept of projectable Jacobi fields from 
\cite{GromWals} one gets a shorter (but more learned) proof of the Corollary.
\end{remark}


\begin{thebibliography}{99}
\bibitem{AlfWal} S. Aloff, N. Wallach, \emph{An infinite family of distinct
7-manifolds admitting positively curved Riemannian structures}, Bull. Amer.
Math. Soc. 81 (1975), 93--97.

\bibitem{Baz} Ya.V. Bazaikin, \emph{On one family of 13-dimensional closed
Riemannian positively curved manifolds}, Sib. Math. J. 37 (1996), 1219--1237.

\bibitem{Berg1} M. Berger, \emph{Les vari\'{e}t\'{e}s riemanniennes \`{a}
courbure positive}, Bull. Soc. Math. Belg. 10 (1958) 88-104.

\bibitem{Cheeg} J. Cheeger, \emph{Some examples of manifolds of nonnegative
curvature}. J. Differential Geometry 8 (1973), 623--628.

\bibitem{Dear} O. Dearricott, \emph{A 7--manifold with positive curvature, }%
Duke Math. J., to appear.

\bibitem{Esch1} J.-H. Eschenburg, \emph{Freie isometrische Aktionen auf
kompakten Lie-Gruppen mit positiv gekr\"{u}mmten Orbitr\"{a}umen}.
Schriftenreihe des Mathematischen Instituts der Universit\"{a}t M\"{u}nster,
2. Serie [Series of the Mathematical Institute of the University of M\"{u}%
nster, Series 2], 32. Universit\"{a}t M\"{u}nster, Mathematisches Institut, M%
\"{u}nster, 1984. vii+177 pp.

\bibitem{Esch3} J.-H. Eschenburg,\emph{\ Inhomogeneous spaces of positive
curvature}. Differential Geom. Appl. 2 (1992), no. 2, 123--132.

\bibitem{Esch2} J.-H. Eschenburg, \emph{New examples of manifolds with
strictly positive curvature}, Invent. Math. 66 (1982), 469--480.

\bibitem{Gray} A. Gray, \emph{Pseudo-Riemannian almost product manifolds and
submersions,} J. Math. Mech. 16 (1967), 715--737.

\bibitem{GromMey} D. Gromoll and W. Meyer, \emph{An exotic sphere with
nonnegative sectional curvature}. Ann. of Math. (2) 100 (1974), 401--406.

\bibitem{GromWals} D. Gromoll and G. Walschap, \emph{Metric Foliations and
Curvature, }Birkh\"{a}user, 2009.

\bibitem{GuiWals} L. Guijarro and G. Walschap, \emph{The metric projection
onto the soul, }Trans. Amer. Math. Soc. 352 (2000), no. 1, 55-69

\bibitem{GuiWals0} L. Guijarro, G.Walschap,\emph{\ When is a Riemannian
submersion homogeneous?}, Geom. Dedicata 125 (2007), 47--52.

\bibitem{GrovVerdZil} K. Grove, L. Verdiani, and W. Ziller, \emph{A
positively curved manifold homeomorphic to }$T_{1}S^{4}$\emph{, }Geom. \&
Funct. Anal., to appear. http://arxiv.org/abs/0809.2304

\bibitem{GrovZil} K.Grove and W.Ziller, \emph{Lifting group actions and
nonnegative curvature}, to appear in Trans. Amer.Math.Soc.
http://arxiv.org/abs/0801.0767

\bibitem{Munt} M. Munteanu, \emph{One-dimensional metric foliations on
compact Lie Groups}, Michigan Math. J. 54 (2006), 25-23.

\bibitem{O'Neill} B. O'Neill, \emph{The fundamental equations of a submersion%
}, Michigan Math. J. 13 (1966), 459--469.

\bibitem{PetWilh2} P. Petersen and F. Wilhelm, \emph{An exotic sphere with
positive sectional curvature, }preprint. http://arxiv.org/abs/0805.0812

\bibitem{PetWilh1} P. Petersen and F. Wilhelm, \emph{Some principals for
deforming nonnegative curvature} preprint. http://arxiv.org/abs/0908.3026

\bibitem{StraWals} M. Strake and G. Walschap, $\Sigma$-\emph{flat manifolds
and Riemannian submersions}, Manuscripta Math. 64 (1989), 213--226

\bibitem{Tapp} K. Tapp, \emph{Flats in Riemannian submersions from Lie
groups, }Asian Journal of Mathematics, Vol. 13, No. 4 (2009), 459-464.

\bibitem{Wal} N. Wallach, \emph{Compact homogeneous Riemannian manifolds
with strictly positive curvature}, Ann. of Math. 96 (1972), 277--295.

\bibitem{Wilk2} B. Wilking, \emph{A duality theorem for Riemannian
foliations in nonnegative sectional curvature}, Geom. Funct. Anal. 17
(2007), 1297--1320.

\bibitem{Wilk1} B. Wilking, \emph{Manifolds with positive sectional
curvature almost everywhere}. Invent. Math. 148 (2002), no. 1, 117--141.
\end{thebibliography}
\end{document}